\documentclass[10pt]{article}
\usepackage[alphabetic,lite]{amsrefs}
\usepackage{amsmath}
\usepackage{amsthm}
\usepackage{amsfonts}
\usepackage{mathrsfs}
\usepackage{amssymb}
\usepackage{amscd}

\usepackage{xcolor}

\usepackage{hyperref}
\hypersetup{
    colorlinks=true,       
    linkcolor=blue,          
    citecolor=magenta,        
    filecolor=green,      
    urlcolor=purple           
}

    \newtheorem{Lem}{Lemma}[section]
    \newtheorem{Prop}[Lem]{Proposition}
    \newtheorem{Thm}[Lem]{Theorem}
    \newtheorem{Cor}[Lem]{Corollary}

\theoremstyle{definition}

    \newtheorem{Rem}[Lem]{Remark}
    \newtheorem{Que}{Question}

\newcommand{\PNk}{{\mathbb{P}_k^{N}}}
\newcommand{\PNz}{{\mathbb{P}_{\mathbb{Z}}^{N}}}

\newcommand{\Pqd}{{\mathbb{P}^{14}_{k}}^\vee}
\newcommand{\ppt}{\mathbb{P}^{(2)}}
\newcommand{\G}{\mathscr{G}}
\newcommand{\bp}{\mathbb{P}^2_k}
\newcommand{\bpd}{{\mathbb{P}^2_k}^\vee}
\newcommand{\Pq}{\mathbb{P}^{14}_k}
\newcommand{\wtP}{\widetilde{\mathbb{P}}_k^{14}}

\newcommand{\ro}[1]{\rho_{#1}}
\newcommand{\Ro}[1]{R_{#1}}
\newcommand{\kl}{q}

\newcommand{\bitq}[2]{\left\langle #1 , #2 \right\rangle}
\newcommand{\GL}{\mathfrak{gl}_3}
\newcommand{\SL}{\mathfrak{sl}_3}
\newcommand{\fer}{{\rm Fer}}
\newcommand{\Z}{\mathbb{Z}}
\newcommand{\Jn}[1]{J_{#1}}
\newcommand{\an}{A_n}
\newcommand{\tn}{t_n}
\newcommand{\coco}[2]{\left[ #1 \,|\, #2 \right]}

\newcommand{\wtH}{H}
\newcommand{\hh}[1]{H_{#1}}
\newcommand{\oh}[1]{\overline{H}_{#1}}

\begin{document}

\title{On a question of Dolgachev}
\author{Marco Pacini and Damiano Testa}

\maketitle

\begin{abstract}
For each even, positive integer~$n$, we define a rational self-map on the space of plane curves of degree~$n$, using classical contravariants.  In the case of plane quartics, we show that the degree of this map is~15.  This answers a question of Dolgachev on the moduli space of curves of genus~3.
\end{abstract}

\bigskip

\noindent
{\small\bf{Mathematics Subject Classification (2010)}}{
\small
14Q05,
14H45,
14H50,
14N20.}

\medskip

\noindent
{\small\bf{Keywords.}}{
\small
Plane curves, quartics, invariants of ternary forms.}

\section*{Introduction}

An inspiring question of Dolgachev motivates the present paper.  First, we describe a classical construction that associates to a general plane quartic curve $C$ another plane quartic curve $\hh{4}(C)$.

Let~$k$ be a field of characteristic relatively prime to $6$ and fix a smooth plane quartic $C \subset \bp$ over $k$.  A line $\ell \subset \bp$, transverse to $C$, intersects $C$ in a configuration of~$4$ distinct points.  The double cover $C_\ell$ of $\ell$ branched above these~$4$ points is a smooth curve of genus one.  We let $\hh{4}(C) \subset \bpd$ denote the closure of the locus of lines $\ell \subset \bp$ such that the $j$-invariant of the curve $C_\ell$ vanishes.  The closed subset $\hh{4}(C) \subset \bpd$ is a plane curve of degree~$4$.  We thus obtain a rational self-map $\hh{4} \colon \Pq \dashrightarrow \Pq$ of the space of plane quartics, assigning to the quartic $C \subset \bp$ the quartic $\hh{4}(C) \subset \bpd$.

\begin{Que}[Dolgachev] \label{q:acca}
Is the rational map $\hh{4}$ {\emph{generically finite}}?  If so, what is its {\emph{degree}}?
\end{Que}

We answer this question in Theorem~\ref{thm:15}: the map $\hh{4}$ is generically finite of degree~$15$.  The second author found the answer to Question~\ref{q:acca}, involving the use of a computer in an essential way.  We propose here a proof that we check entirely by hand.  Nevertheless, determining the degree is still the outcome of a lengthy computation: we do not have an interpretation for the fibers of the map~$\hh{4}$.

In an attempt to obtain a more conceptual understanding of the fibers of the map~$\hh{4}$, we computed the monodromy group of~$\hh{4}$.  Theorem~\ref{thm:monodromia} shows that this group is the full symmetric group~$\mathfrak{S}_{15}$.  Unfortunately we have not been able to use this information: we still lack an understanding of how the~$15$ quartics in a general fiber of~$\hh{4}$ arise.

To put Question~\ref{q:acca} into perspective, observe that the map~$\hh{4}$ is equivariant with respect to the group $\mathbb{P}GL_3$ of projective changes of coordinates.  Moreover, the quotient of the space of plane quartic curves by the group $\mathbb{P}GL_3$ is birational to the moduli space $\mathcal{M}_3$ of curves of genus~$3$.  After checking that the map $\hh{4}$ is generically finite, we deduce that it descends to a generically finite rational map $\oh{4} \colon \mathcal{M}_3 \dashrightarrow \mathcal{M}_3$.  We show that the degree of~$\oh{4}$ is also~$15$.

More generally, Dolgachev considers rational self-maps of moduli spaces of curves of low genus (such as~$\oh{4}$) and of hypersurfaces (such as~$\hh{4}$) in~\cite{dolg}.  He provides several examples and constructions of dominant rational self-maps of such spaces to themselves of degree strictly larger than~$1$.

In line with Dolgachev's general strategy, for each even degree $n$, we introduce a rational self-map
\[
\hh{n} \colon \mathbb{P}^{\binom{n+2}{2}-1} \dashrightarrow \mathbb{P}^{\binom{n+2}{2}-1}
\]
on the projective space of plane curves of degree $n$ (for odd~$n$, the map is not defined).  As in the case $n=4$, the map~$\hh{n}$ is equivariant with respect to the group $\mathbb{P}GL_3$.  When $\hh{n}$ is generically finite, it descends to a generically finite rational self-map $\oh{n}$ on the moduli space of plane curves of degree~$n$.  We can extend Question~\ref{q:acca} to any even positive integer~$n$.

\begin{Que} \label{q:hn}
Is the rational map $\hh{n}$ {\emph{generically finite}}?  If so, what are the {\emph{degrees}} of $\hh{n}$ and $\oh{n}$?
\end{Que}

For $n=2$, the map $\hh{2}$ assigns to a general conic its dual conic and is therefore birational.  For $n=4$, the map $\hh{4}$ is the one defined above, of degree~$15$.   For $n \in \{2,4\}$, the degrees of $\hh{n}$ and $\oh{n}$ coincide: we do not know if they are always equal.

In Section~\ref{sec:pre}, we briefly set up the notation and basic facts about invariants, contravariants and Lie algebras for ternary forms.  Here, we define the rational maps~$\hh{n}$ on the space of plane curves of degree~$n$ to itself.  We recommend~\cite{cag}*{Chapter~1} as a general introduction to the topic; the whole book contains a wealth of information and details about this beautiful subject and beyond.  In Equation~\eqref{eq:tril}, we introduce the trilinear form~$\tn(-,-,-)$: it provides a fundamental link between an invariant of ternary forms and the maps~$\hh{n}$ via the Lie algebra $\SL$.  The heart of the Section is devoted to the proof of Theorem~\ref{thm:tsim}, providing a crucial symmetry of the trilinear form~$\tn$.  In Section~\ref{sec:grado}, we determine explicitly the scheme-theoretic fiber of~$\hh{4}$ over the Fermat quartic (see Theorem~\ref{thm:fibfer}).  Theorem~\ref{thm:15} exploits the structure of this fiber: the degree of~$\hh{4}$ is~$15$ and the monodromy group of $\hh{4}$ contains the alternating group~$\mathfrak{A}_{15}$ (see also Theorem~\ref{thm:monodromia}).  Knowing that the monodromy group is $2$-transitive, allows us to prove that the degree of the induced quotient map~$\oh{4}$ is also~$15$.  In Section~\ref{sec:inva}, we define an invariant $\ro{n}$ vanishing on the locus where the differential of the map~$\hh{n}$ is not an isomorphism.  In the case $n=4$, the polynomial~$\ro{4}$ is an invariant of degree~$15$ of ternary quartic forms: Theorem~\ref{thm:lrs} gives the expression of $\ro{4}$ in terms of the Dixmier-Ohno invariants.  The proof uses in an essential way the Magma~\cite{magma} package {\tt{g3twists}} described in~\cite{lrs}.  We include in the source of the file~\cite{PTa}, the code that computes and checks our assertions.  In the Appendix, we determine the degree of a scheme that plays an important role in our argument.  The calculation could alternatively be carried out using a computer.

\subsection*{Acknowledgments}
We wish to thank Igor Dolgachev and Jeroen Sijsling for interesting conversations during the preparation of this paper.  We also thank the anonymous referee for suggestions improving the quality of the paper.  The first author was partially supported by CNPq, processo 200377/2015-9 and processo 301314/2016-0.

\section{Preliminary identities}
\label{sec:pre}

In this section, we prove the main identities used in the paper working over the ring~$\Z$ of integers.  In the later sections, we specialize to fields of characteristic coprime to~$6$.

Let $x,y,z$ be homogeneous coordinates on the projective plane $\mathbb{P}^2_{\Z}$ and let $u,v,w$ be the dual coordinates on the dual projective plane.  Let $n$ be a non-negative integer and denote by $\Z [x,y,z]_n$ and $\Z [u,v,w]_n$ the $\Z$-modules of ternary forms of degree $n$ in respective variables $x,y,z$ and $u,v,w$.  Set $N = \binom{n+2}{2}-1$; thus the rank of the free $\Z$-module $\Z [x,y,z]_n$ is $N+1$.  We identify $\mathbb{P}(\Z [x,y,z]_n)$ and $\mathbb{P}(\Z [u,v,w]_n)$ with $\PNz$ equivariantly with respect to the action of the group-scheme $GL_{3,\Z}$.  The projective space $\PNz$ is the space of plane curves of degree~$n$.

We define the {\emph{polar pairing}} $\bitq{-}{-} \colon \Z[u,v,w]_n \times \Z[x,y,z]_n \to \Z$ by
\[
\bitq{q_1}{q_2} = q_1(\partial_x , \partial_y , \partial_z) \, q_2(x,y,z) .
\]
Equivalently, the polar pairing $\bitq{-}{-}$ is the unique bilinear form taking the values
\begin{equation} \label{eq:popa}
( u^{a_1} v^{b_1} w^{c_1} \, , \, x^{a_2} y^{b_2} z^{c_2} ) \longmapsto
\begin{cases}
a_1! \, b_1! \, c_1! & \textrm{if } (a_1,b_1,c_1) = (a_2,b_2,c_2), \\[3pt]
0 & \textrm{otherwise,}
\end{cases}
\end{equation}
on monomials.  The polar pairing allows us to associate to each ternary form $q \in \Z[u,v,w]_n$ a linear form $\bitq{q}{-} \colon \Z [x,y,z]_n \longrightarrow \Z$.

We define a bilinear map
\[
\Jn{n} \colon \Z[x,y,z]_n \times \Z[x,y,z]_n
\longrightarrow
\Z[u,v,w]_n
\]
as follows:
\begin{equation} \label{eq:defjn}
\Jn{n} (q_1(x,y,z) , q_2(x,y,z)) =
q_1 \left(
\begin{vmatrix}
\partial_y & \partial_z \\
v & w
\end{vmatrix}
,
-\begin{vmatrix}
\partial_x & \partial_z \\
u & w
\end{vmatrix}
,
\begin{vmatrix}
\partial_x & \partial_y \\
u & v
\end{vmatrix}
\right)
q_2(x,y,z) .
\end{equation}

For a ternary form $q(x,y,z) \in \Z[x,y,z]_n$ of degree $n$, we let $\hh{n}(q)(u,v,w)$ be the ternary form
\[
\hh{n}(q)(u,v,w) = \Jn{n}(q(x,y,z),q(x,y,z))
\]
of degree $n$; we call $\hh{n} (q)(u,v,w)$ the {\emph{harmonic form}} associated to $q$.  If the integer $n$ is odd, then $\hh{n}(q)$ is always zero.  The function $\hh{n} \colon \Z [x,y,z]_n \to \Z[u,v,w]_n$ is a contravariant of ternary forms with respect to the group-scheme $SL_{3,\Z}$ (see~\cite{cag}*{Section~3.4.2 and Example~3.4.2}).

Combining the polar pairing with the bilinear map $\Jn{n}$, we define a trilinear form $\tn \colon (\Z [x,y,z]_n)^{\times 3} \to \Z$:
\begin{equation} \label{eq:tril}
\tn (q_1,q_2,q_3) = \bitq{\Jn{n} (q_1,q_2)}{q_3}.
\end{equation}

\begin{Thm} \label{thm:tsim}
For every permutation $\sigma \in \mathfrak{S}_3$, the trilinear form $\tn$ satisfies
\[
\tn (q_{\sigma(1)},q_{\sigma(2)},q_{\sigma(3)}) = ({\rm{sign}}\, \sigma)^n \tn (q_1,q_2,q_3),
\]
where ${\rm{sign}}\, \sigma$ denotes the sign of the permutation $\sigma$.
\end{Thm}

We now give a proof of Theorem~\ref{thm:tsim}.  The argument is entirely combinatorial and relies on a few identities that we prove first.  We use an alternative definition of the bilinear map~$\Jn{n}$.

Let $m_1,m_2 \in \Z [x,y,z]_n$ be monomials and set $m_3 = \frac{(xyz)^n}{m_1m_2} \in \Z [x^{\pm 1} , y^{\pm 1} , z^{\pm 1}]$.  For $i \in \{ 1 , 2 , 3 \}$, write $m_i = x^{a_i} y^{b_i} z^{c_i}$; by construction, the integers $a_1$, $b_1$, $c_1$, $a_2$, $b_2$, $c_2$ are non-negative, while $a_3,b_3$ and $c_3$ need not be.  Define an integer $\coco{m_1}{m_2}$ by
\begin{equation} \label{e:defm}
\coco{m_1}{m_2} =
(-1)^{a_2 + b_3 + c_1} \,
a_1! \, b_1! \, c_1! \,
\sum (-1)^{\alpha}
\binom{a_2}{\alpha}
\binom{b_2}{c_1-\alpha}
\binom{c_2}{b_3-\alpha}
,
\end{equation}
and observe that the summands vanish if $\alpha$ is outside of the range $0 \leq \alpha \leq \min \{ a_2 , c_1\}$.

\begin{Prop}[Properties of $\coco{-}{-}$] \label{prop:coco}
Let $m_1=x^{a_1} y^{b_1} z^{c_1},m_2=x^{a_2} y^{b_2} z^{c_2} \in \Z [x,y,z]_n$ be monomials.  Set $m_3=\frac{(xyz)^n}{m_1 m_2} = x^{a_3} y^{b_3} z^{c_3}$, with $a_3, b_3, c_3$ integers.
\renewcommand{\theenumi}{\roman{enumi}}
\begin{enumerate}
\item
\label{en:bam}
$\coco{1}{1} = 1$.
\item
\label{en:mric}
If $x$ divides $m_1$, then
$\coco{m_1}{m_2} =
b_2 \coco{\frac{m_1}{x}}{\frac{m_2}{y}}
- c_2 \coco{\frac{m_1}{x}}{\frac{m_2}{z}}$.

\noindent
If $y$ divides $m_1$, then
$\coco{m_1}{m_2} =
c_2 \coco{\frac{m_1}{y}}{\frac{m_2}{z}}
- a_2 \coco{\frac{m_1}{y}}{\frac{m_2}{x}}$.

\noindent
If $z$ divides $m_1$, then
$\coco{m_1}{m_2} =
a_2 \coco{\frac{m_1}{z}}{\frac{m_2}{x}}
- b_2 \coco{\frac{m_1}{z}}{\frac{m_2}{y}}$.
\item
\label{en:msim}
$\coco{m_2}{m_1} = (-1)^n \coco{m_1}{m_2}$.
\item
\label{en:msva}
If at least one of $a_3,b_3,c_3$ is strictly negative, then, for every $\alpha$, the product $\binom{a_2}{\alpha}
\binom{b_2}{c_1-\alpha}
\binom{c_2}{b_3-\alpha}$ vanishes.  In particular, the identity $\coco{m_1}{m_2} = 0$ holds.
\item
\label{en:mtris}
If the integers $a_3,b_3,c_3$ are non-negative, then the identity
\[
a_3! b_3! c_3! \coco{m_1}{m_2} = (-1)^n a_1! b_1! c_1! \coco{m_3}{m_2}
\]
holds.
\end{enumerate}
\end{Prop}

\begin{proof}
\eqref{en:bam} Follows directly from the definition.

\eqref{en:mric}
We only argue the case in which $x$ divides $m_1$; the remaining cases are analogous.  The identities
\[
b_2 \binom{b_2-1}{c_1-\alpha} = (b_2 - c_1 + \alpha)
\binom{b_2}{c_1-\alpha}
\qquad
\quad
c_2 \binom{c_2-1}{b_3-1-\alpha}
=
(b_3 - \alpha)
\binom{c_2}{b_3-\alpha}
\]
and
\[
b_2 - c_1 + \alpha + b_3 - \alpha = a_1
\]
hold.  Combining these identities with Equation~\eqref{e:defm} we find
\[
\begin{array}{rcl}
b_2 \coco{\frac{m_1}{x}}{\frac{m_2}{y}} - c_2 \coco{\frac{m_1}{x}}{\frac{m_2}{z}} =
\\[12pt]
& &
\hspace{-120pt}
= \displaystyle
(-1)^{a_2 + b_3 + c_1} \,
(a_1-1)! \, b_1! \, c_1! \,
\sum (-1)^{\alpha}
\binom{a_2}{\alpha}
b_2
\binom{b_2-1}{c_1-\alpha}
\binom{c_2}{b_3-\alpha} \\[12pt]
& &
\hspace{-113pt}
\displaystyle
\hspace{-8pt}
- (-1)^{a_2 + b_3 -1 + c_1} \,
(a_1-1)! \, b_1! \, c_1! \,
\sum (-1)^{\alpha}
\binom{a_2}{\alpha}
\binom{b_2}{c_1-\alpha}
c_2
\binom{c_2-1}{b_3-1-\alpha} \\[12pt]
& &
\hspace{-120pt}
= \displaystyle
(-1)^{a_2 + b_3 + c_1} \,
a_1! \, b_1! \, c_1! \,
\sum (-1)^{\alpha}
\binom{a_2}{\alpha}
\binom{b_2}{c_1-\alpha}
\binom{c_2}{b_3-\alpha} \\[12pt]
& &
\hspace{-120pt}
= \coco{m_1}{m_2}
\end{array}
\]
as needed.

\eqref{en:msim}
Let $\alpha$ be any integer and set $\alpha' = b_3 - \alpha$.  Expanding the binomials in the product $a_1! \, b_1! \, c_1! \,
\binom{a_2}{\alpha}
\binom{b_2}{c_1-\alpha}
\binom{c_2}{b_3-\alpha}$
we find
\[
%
\frac
{a_1! \, b_1! \, c_1! \, a_2! \, b_2! \, c_2!}
{\alpha! (a_2-\alpha)!
(c_1-\alpha)! (b_2 - c_1 + \alpha)!
(b_3-\alpha)! (c_2 - b_3 + \alpha)!}.
\]
We use the equalities
\[
a_2 - b_3 + \alpha' = b_1 - c_2 + \alpha'
\qquad
\qquad
b_2 - c_1 + b_3 - \alpha' = a_1 - \alpha'
\]
to find
\begin{equation} \label{eq:1f2f}
a_1! \, b_1! \, c_1! \,
\binom{a_2}{\alpha}
\binom{b_2}{c_1-\alpha}
\binom{c_2}{b_3-\alpha}
=
a_2! \, b_2! \, c_2! \,
\binom{a_1}{\alpha'}
\binom{b_1}{c_2-\alpha'}
\binom{c_1}{b_3-\alpha'}.
\end{equation}
Finally, from the identity
\[
a_2 + b_3 + c_1 = n - a_1 + 2 b_3 - c_2
\]
we deduce
\begin{equation} \label{eq:1s2s}
(-1)^{a_2 + b_3 + c_1 + \alpha} = (-1)^n (-1)^{a_1 + b_3 + c_2 - \alpha'}.
\end{equation}
Combining Equations~\eqref{eq:1f2f} and~\eqref{eq:1s2s} and summing over all $\alpha$, we obtain the required identity
\[
\coco{m_1}{m_2} =(-1)^n \coco{m_2}{m_1}.
\]

\eqref{en:msva}
If $\alpha$ is not in the interval $[0,\min \{ a_2 , c_1\}]$, then the product $\binom{a_2}{\alpha} \binom{b_2}{c_1-\alpha}$ vanishes.  Suppose therefore that $\alpha$ satisfies the inequalities $0 \leq \alpha \leq \min \{ a_2 , c_1\}$.
\begin{itemize}
\item
If $a_3<0$, then, using $c_2-b_3 = a_3-c_1$, we obtain that $\binom{c_2}{b_3-\alpha} = \binom{c_2}{c_2-b_3+\alpha} = \binom{c_2}{a_3-c_1+\alpha}$ vanishes.
\item
If $b_3<0$, then $\binom{c_2}{b_3-\alpha}$ vanishes.
\item
If $c_3<0$, then, using $b_2-c_1 = c_3-a_2$, we obtain that  $\binom{b_2}{c_1-\alpha} = \binom{b_2}{b_2-c_1+\alpha} =\binom{b_2}{c_3+\alpha-a_2}$ vanishes.
\end{itemize}
The vanishing of $\coco{m_1}{m_2}$ follows, since we just proved that every summand in Equation~\eqref{e:defm} is zero.

\eqref{en:mtris}
Using the identities $c_2-b_3 = b_1-a_2$ and $b_2-c_1 = c_3-a_2$, we find, for any integer $\alpha$, the equalities
\[
\binom{c_2}{b_3-\alpha} =
\binom{c_2}{c_2-b_3+\alpha} =
\binom{c_2}{b_1-(a_2-\alpha)}
\]
and
\[
\binom{b_2}{c_1-\alpha} =
\binom{b_2}{b_2-c_1+\alpha} =
\binom{b_2}{c_3-(a_2-\alpha)}.
\]
Moreover, also the equality $b_3+c_1-(b_1+c_3) = n+a_2 -2(b_1+c_3)$ holds.  Set $\alpha' = a_2-\alpha$; we deduce the equality
\begin{multline*}
(-1)^{a_2 + b_3 + c_1 + \alpha}
\binom{a_2}{\alpha}
\binom{b_2}{c_1-\alpha}
\binom{c_2}{b_3-\alpha}
=
\\
=
(-1)^n
(-1)^{a_2 + b_1 + c_3 + \alpha'}
\binom{a_2}{\alpha'}
\binom{b_2}{c_3-\alpha'}
\binom{c_2}{b_1-\alpha'}
\end{multline*}
and we conclude summing over all $\alpha$.
\end{proof}

Let $m_1,m_2$ be monomials in $\Z [x,y,z]_n$.  The ratio $\frac{(xyz)^n}{m_1m_2}$ is a Laurent monomial in $\Z [x^{\pm 1} , y^{\pm 1} , z^{\pm 1}]$ and we let $a_3,b_3,c_3 \in \Z$ be its exponents: $\frac{(xyz)^n}{m_1m_2} = x^{a_3} y^{b_3} z^{c_3}$.

We define a bilinear map $\Jn{n}' \colon \Z[x,y,z]_n \times \Z[x,y,z]_n \to  \Z[u,v,w]_n$ by setting
\[
\Jn{n}'(m_1,m_2) = \coco{m_1}{m_2} u^{a_3} v^{b_3} w^{c_3}
\]
on monomials and extending by bilinearity.

\begin{Cor} \label{cor:jnjn}
For all non-negative integers $n$, the bilinear maps $\Jn{n}$ and $\Jn{n}'$ coincide.
\end{Cor}

\begin{proof}
We proceed by induction on $n$.  The case $n=0$ follows from the definitions.  Suppose that $n>0$ is an integer and that the maps $\Jn{n-1}$ and $\Jn{n-1}'$ coincide.  Let $m_1,m_2$ be monomials in $\Z[x,y,z]_n$.  To prove the result, it suffices to show that the identity $\Jn{n}(m_1,m_2) = \Jn{n}'(m_1,m_2)$ holds.  Let $m_3$ denote the Laurent monomial $\frac{(xyz)^n}{m_1m_2}$.  For $i \in \{1,2,3\}$, write $m_i = x^{a_i} y^{b_i} z^{c_i}$.  If $x$ divides $m_1$, that is, if $a_1$ is strictly positive, then Equation~\eqref{eq:defjn} implies the identity
\[
\Jn{n} (m_1,m_2) =
b_2 \Jn{n-1}
\left(
\frac{m_1}{x} ,
\frac{m_2}{y}
\right) w
-
c_2\Jn{n-1}
\left(
\frac{m_1}{x} ,
\frac{m_2}{z}
\right) v,
\]
with the convention that if $b_2$ or $c_2$ vanish, then the corresponding term vanishes as well.  Using the inductive hypothesis, we obtain
\[
\Jn{n} (m_1,m_2) = b_2
\coco{\frac{m_1}{x}}{\frac{m_2}{y}} u^{a_3}v^{b_3}w^{c_3-1}
w
-
c_2
\coco{\frac{m_1}{x}}{\frac{m_2}{z}} u^{a_3}v^{b_3-1}w^{c_3}
v
\]
and this last expression equals $\Jn{n}'(\frac{m_1}{x},\frac{m_2}{z})$ by Proposition~\ref{prop:coco}~\eqref{en:mric}.  Arguing similarly if $y$ or $z$ divides $m_1$, we conclude the proof of the induction step.  The result follows by induction.
\end{proof}

\begin{proof}[Proof of Theorem~\ref{thm:tsim}]
Let $q_1,q_2,q_3 \in \Z[x,y,z]_n$ be three forms.  By Corollary~\ref{cor:jnjn} and Proposition~\ref{prop:coco}~\eqref{en:msim}, the identity
\[
\Jn{n} (q_2,q_1) = (-1)^n \Jn{n}(q_1,q_2)
\]
holds.  In particular, to prove the result, it suffices to show the identity
\[
\tn (q_1,q_2,q_3) = \tn(q_2,q_3,q_1).
\]
Using the linearity of $\tn$ in its three arguments, it suffices to prove the result in the case in which $q_1,q_2,q_3$ are monomials.  For $i \in \{1,2,3\}$, write $q_i = x^{a_i} y^{b_i} z^{c_i}$.  Using the definition of the polar pairing, we deduce that $\tn(q_1,q_2,q_3)$ vanishes if the monomial $q_1 q_2 q_3$ is not equal to $(xyz)^n$.  Thus, suppose that $q_1 q_2 q_3$ equals $(xyz)^n$.  Corollary~\ref{cor:jnjn} allows us to deduce the equality $\Jn{n} (q_1,q_2) = \coco{q_1}{q_2} u^{a_3} v^{b_3} w^{c_3}$.  We compute
\[
\begin{array}{lrcl}
{\scriptstyle{\textrm{Equation~\eqref{eq:tril}}}}
& \tn (q_1,q_2,q_3) & = & \coco{q_1}{q_2}
\bitq{u^{a_3} v^{b_3} w^{c_3}} {x^{a_3} y^{b_3} z^{c_3}}
\\[3pt]
{\scriptstyle{\textrm{Equation~\eqref{eq:popa}}}}
&& = &
a_3! b_3! c_3! \coco{q_1}{q_2} \\[3pt]
{\scriptstyle{\textrm{Proposition~\ref{prop:coco}~\eqref{en:mtris}}}}
&& = &
(-1)^n a_1! b_1! c_1! \coco{q_3}{q_2} \\[3pt]
{\scriptstyle{\textrm{Corollary~\ref{cor:jnjn}}}}
&& = &
\bitq{(-1)^n \Jn{n}(q_3,q_2)}{q_1} \\[3pt]
{\scriptstyle{\textrm{Proposition~\ref{prop:coco}~\eqref{en:msim}}}}
&& = &
\bitq{\Jn{n}(q_2,q_3)}{q_1} \\[3pt]
{\scriptstyle{\textrm{Equation~\eqref{eq:tril}}}}
&& = &
\tn(q_2,q_3,q_1)
\end{array}
\]
and we are done.
\end{proof}

We make use of the relationship between the contravariant $\hh{n}$ and an invariant $\an$ under the action of $SL_{3,\Z}$.  The expression
\[
\an (q) = \tn (q,q,q) = \bitq{\hh{n} (q)}{q} \in \Z
\]
is an invariant of ternary forms $q$ under $SL_{3,\Z}$; the degree of $\an$ in the coefficients of the form $q$ is~$3$.  If $n$ is odd, then the harmonic form $\hh{n}$ vanishes identically; therefore, the same is true of the invariant $\an$.  Denote by $\GL$ and by $\SL$ the Lie algebras of $GL_{3,\Z}$ and $SL_{3,\Z}$ respectively.

\begin{Thm}\label{thm:lienu}
For every derivation $\mathfrak{g}$ in $\SL$, the identity
\begin{equation} \label{eq:lierel}
\bitq{\hh{n}(q)}{\mathfrak{g}q} = 0,
\end{equation}
holds.
\end{Thm}

\begin{proof}
If $\mathfrak{g}$ is a derivation in $\GL$ and $q_1,q_2,q_3$ are forms in $\Z[x,y,z]_n$, then the equality
\[
\mathfrak{g} \tn (q_1,q_2,q_3) = \bitq{\Jn{n}(\mathfrak{g}q_1,q_2)}{q_3} + \bitq{\Jn{n}(q_1,\mathfrak{g}q_2)}{q_3} + \bitq{\Jn{n}(q_1,q_2)}{\mathfrak{g}q_3}
\]
holds.  Using Theorem~\ref{thm:tsim}, we obtain the identity $\mathfrak{g} \an (q) = 3 \bitq{\Jn{n}(q,q)}{\mathfrak{g}q}$.  Suppose now that $\mathfrak{g}$ is in $\SL$.  Since $\an$ is invariant under $SL_{3,\Z}$, we deduce that $\mathfrak{g} \an (q) = 0$.  Combining these formulas, we obtain the required identity.
\end{proof}

\section{The computation of the degree}
\label{sec:grado}

We now restrict our attention to the case $n=4$.  To perform the main computations, we work over a general field~$k$ of characteristic zero.  An easy argument appearing in Remark~\ref{rem:chp0} shows that this restriction on the characteristic can be weakened.

The invariant $A_4(q)$ has degree~$3$ in the coefficients of~$q$: it is, up to scaling, the unique non-constant invariant of smallest degree of plane quartics.
Salmon denotes the contravariant $\hh{4}$ by~$\sigma$~\cite{sal}*{p.~264, \S292} and the invariant~$A_4(q)$ by~$A$~\cite{sal}*{p.~264, \S293}; Dolgachev denotes the contravariant $\hh{4}$ by~$\Omega_{2,4}$ and the invariant~$A_4(q)$ by~$I_3$.

Let $\fer \subset \bp$ denote the Fermat quartic with equation $\fer \colon x^4+y^4+z^4 = 0$; similarly, let $\fer' \subset \bpd$ denote the Fermat quartic with equation $\fer' \colon u^4+v^4+w^4 = 0$.  We also define the four quartics $C_0,C_1,C_2,C_3 \subset \bp$ with equations
\[
\begin{array}{crcl}
C_0 \colon &
(x^4+y^4+z^4) - 6 (x^2 y^2 + x^2 z^2 + y^2 z^2)
& = & 0
\\[4pt]
C_1 \colon &
(x^4+y^4+z^4) - 6 (x^2 y^2 - x^2 z^2 - y^2 z^2)
& = & 0
\\[4pt]
C_2 \colon &
(x^4+y^4+z^4) - 6 (-x^2 y^2 + x^2 z^2 - y^2 z^2)
& = & 0
\\[4pt]
C_3 \colon &
(x^4+y^4+z^4) - 6 (-x^2 y^2 - x^2 z^2 + y^2 z^2)
& = & 0.
\end{array}
\]
The curves $C_0,C_1,C_2,C_3$ are all isomorphic: permutations of the coordinates induce projective equivalences among $C_1,C_2,C_3$; rescaling $z$ by a square root of ${-1}$ transforms $C_0$ into $C_1$.  An easy check shows that they are smooth.

Let $\wtP \subset \Pq \times \Pqd$ be the closure of the graph of the rational map~$\hh{4}$.  The second projection $\Pq \times \Pqd \to \Pqd$ restricts to a morphism
\[
\wtH \colon \wtP \longrightarrow \Pqd .
\]
A plane quartic~$C$ in the indeterminacy locus of~$\hh{4}$ must be singular: see~\cite{PTr}*{Proposition~2.5} for a more precise statement.  Let $(C,D)$ be a pair in~$\wtP$.  If the rational map~$\hh{4}$ is defined at~$C$, then the projection $\Pq \times \Pqd \to \Pq$ restricts to an isomorphism on an open subset of $\wtP$ containing $(C,D)$.  When this happens, to simplify the notation, we identify the pair $(C,D) \in \wtP$ with $C$, since $D$ can be obtained as~$\hh{4}(C)$.

\begin{Thm} \label{thm:fibfer}
The fiber of the morphism~$\wtH$ above the Fermat quartic $\fer' \subset \bpd$ consists of the five quartics $\fer$, $C_0$, $C_1$, $C_2$, $C_3 \subset \bp$, where the Fermat quartic $\fer$ appears with multiplicity~$11$ and each one of the remaining four quartics appears with multiplicity~$1$.
\end{Thm}

\begin{proof}
Let $C$ be a plane quartic and let $q(x,y,z) = \sum a_{ijk} x^i y^j z^k \in k[x,y,z]$ be an equation for $C$.  If the pair $(C,\fer')$ is contained in $\wtP$, then Theorem~\ref{thm:lienu} implies that, for every element $\mathfrak{g}$ of $\SL$, the identity $\bitq{u^4+v^4+z^4}{\mathfrak{g} q} = 0$ holds.  Specializing this identity with $\mathfrak{g}$ in the list
\[
\begin{array}
{ll@{\hspace{40pt}}ll@{\hspace{40pt}}ll@{\hspace{40pt}}ll}
1. &
x \partial_x - y \partial_y
&
3. &
y \partial_x
&
5. &
x \partial_y
&
7. &
x \partial_z
\\[5pt]
2. &
y \partial_y - z \partial_z
&
4. &
z \partial_x
&
6. &
z \partial_y
&
8. &
y \partial_z,
\end{array}
\]
we obtain the identities
\begin{equation*}
\begin{array}
{l@{\hspace{5pt}}l@{\hspace{20pt}}l@{\hspace{5pt}}l@{\hspace{20pt}}l@{\hspace{5pt}}l@{\hspace{20pt}}l@{\hspace{5pt}}l}
1. &
48 (a_{400} - a_{040}) = 0
&
3. &
12 a_{130} = 0
&
5. &
12 a_{013} = 0
&
7. &
12 a_{301} = 0
\\[5pt]
2. &
48 (a_{040} - a_{004}) = 0
&
4. &
12 a_{103} = 0
&
6. &
12 a_{310} = 0
&
8. &
12 a_{031} = 0.
\end{array}
\end{equation*}
We deduce that $q$ is of the form
\[
q(x,y,z) = \rho (x^4+y^4+z^4) + \sigma_{3} x^2 y^2 + \sigma_{2} x^2 z^2 + \sigma_{1} y^2 z^2 + xyz (\tau_{1} x + \tau_{2} y + \tau_{3} z) ,
\]
where $\rho = a_{400} = a_{040} = a_{004}$, $\sigma_1 = a_{022}$, $\sigma_2 = a_{202}$, $\sigma_3 = a_{220}$ and $\tau_1 = a_{211}$, $\tau_2 = a_{121}$, $\tau_3 = a_{112}$.  The pair $(C,\fer')$ is contained in~$\wtP$ if $\hh{4}(q)$ is an equation for~$\fer'$.  Using the expression that we obtained for $q$ we find
\begin{eqnarray*}
\hh{4}(q) & = &
(12 \rho^2 + \sigma_{1}^2) u^4 + (12 \rho^2 + \sigma_{2}^2) v^4 + (12 \rho^2 + \sigma_{3}^2) w^4 \\[3pt]
& &
-2 (\sigma_{3} \tau_{1} v w^3 + \sigma_{3} \tau_{2} u w^3 +
\sigma_{2} \tau_{3} u v^3 + \sigma_{2} \tau_{1} v^3 w +
\sigma_{1} \tau_{2} u^3 w + \sigma_{1} \tau_{3} u^3 v) \\[3pt]
& &
+(12 \rho \sigma_{1} + 2 \sigma_{2} \sigma_{3} + \tau_{1}^2) v^2 w^2 +
(12 \rho \sigma_{2} + 2 \sigma_{1} \sigma_{3} + \tau_{2}^2) u^2 w^2
\\[3pt]
& &
+ (12 \rho \sigma_{3} + 2 \sigma_{1} \sigma_{2} + \tau_{3}^2) u^2 v^2 \\[3pt]
& &
+
(4 \sigma_{1} \tau_{1} - \tau_{2} \tau_{3}) u^2 v w +
(4 \sigma_{2} \tau_{2} - \tau_{1} \tau_{3}) u v^2 w +
(4 \sigma_{3} \tau_{3} - \tau_{1} \tau_{2}) u v w^2 .
\end{eqnarray*}
The condition that $u^4+v^4+w^4$ and $\hh{4}(q)$ be proportional determines a subscheme~$F_0$ of $\mathbb{P}^{14}$.  The scheme~$F_0$ is isomorphic to the scheme~$F$ defined in~\eqref{eq:fibf} and the result follows from Lemma~\ref{lem:groeb}.
\end{proof}

\begin{Rem}
The Fermat curve $\fer$ is not isomorphic to any one of the curves $C_0,C_1,C_2,C_3$.  This is an immediate consequence of Theorem~\ref{thm:fibfer}: the map~$\hh{4}$ is contravariant and hence projectively equivalent curves appear with the same multiplicity in fibers of~$\hh{4}$.
\end{Rem}

We want to compute the monodromy of the morphism~$\wtH$.  For this, we use the following result, due to Jordan (see~\cite{isa}*{Theorem~8.23}).  For a positive integer~$n$, denote by~$\mathfrak{S}_n$ the symmetric group on~$n$ elements and by~$\mathfrak{A}_n$ the alternating group.

\begin{Thm}[Jordan] \label{thm:jordan}
Let~$n$ be a positive integer and let~$G$ be a primitive subgroup of~$\mathfrak{S}_n$.  If~$p$ is a prime satisfying $p < n-2$ and~$G$ contains a $p$-cycle, then~$G$ contains~$\mathfrak{A}_n$.
\end{Thm}

\begin{Thm} \label{thm:15}
The morphism~$\wtH \colon \wtP \to \Pqd$ is generically finite of degree~$15$.  The monodromy group of~$\wtH$ contains the alternating group~$\mathfrak{A}_{15}$.
\end{Thm}

\begin{proof}
Let~$F$ be the fiber of~$\wtH$ over the Fermat quartic curve~$\fer' \subset \bpd$.  By Theorem~\ref{thm:fibfer}, the scheme~$F$ is finite of degree~$15$.  Thus, the morphism~$\wtH$ is quasi-finite in a neighbourhood of~$\fer'$ and therefore finite, since it is projective.  To conclude that the degree of~$\wtH$ is~$15$, it suffices to argue that~$\wtH$ is flat at~$\fer'$.  By the Miracle Flatness Theorem~\cite{mats}*{Theorem~23.1}, it is enough to check that $\wtP$ is smooth at~$F$.  This is true, since the rational map~$\hh{4}$ is defined at the points of~$F$ and its domain,~$\Pq$, is smooth (recall that the graph morphism is an immersion, see~\cite{gw}*{Proposition~9.5}).

By what we just proved, the monodromy group of~$\wtH$ is isomorphic to a subgroup~$G$ of the symmetric group~$\mathfrak{S}_{15}$.  Since~$\wtP$ is irreducible, the group~$G$ is transitive.  Since the fiber~$F$ contains four reduced points and one point of multiplicity~$11$, we deduce that~$G$ contains a subgroup with an orbit consisting of~$11$ elements.  Hence, $G$ also contains a cycle of length~$11$ and is therefore primitive (see~\cite{isa}*{Lemma~8.20}).  Theorem~\ref{thm:jordan} shows that $G$ contains the alternating group~$\mathfrak{A}_{15}$ and we are done.
\end{proof}

\begin{Rem}
Denote by $\PNk /\!\!/ \mathbb{P}GL_3$ the GIT-quotient of $\PNk$ by $\mathbb{P}GL_3$.  We check that, for even $n$, the contravariant $\hh{n}$ induces a rational map
\[
\hh{n} \colon \PNk \dashrightarrow \PNk,
\]
descending to a rational map on the quotient
\[
\oh{n} \colon \PNk /\!\!/ \mathbb{P}GL_3 \dashrightarrow \PNk /\!\!/ \mathbb{P}GL_3.
\]
Indeed, it suffices to find a ternary form $q$ of even degree $n$ defining a plane curve, such that $\hh{n}(q)$ a GIT-stable curve of degree $n$.  For this, we compute
\[
\hh{n} (x^n+y^n+z^n)
= (1+(-1)^n) \cdot n! (u^n + v^n + w^n) = 2 \cdot n! (u^n + v^n + w^n),
\]
and we are done, since smooth curves are GIT-stable.
\end{Rem}

In the case $n=4$, the quotient $\Pq /\!\!/ \mathbb{P}GL_3$ is birational to the moduli space $\mathcal{M}_3$ of smooth curves of genus~$3$ and we obtain
\[
\oh{4} \colon \mathcal{M}_3
\dashrightarrow
\mathcal{M}_3.
\]

\begin{Cor} \label{cor:hbar}
The rational map~$\oh{4}$ is generically finite of degree~$15$.
\end{Cor}

\begin{proof}
Denote by $\ppt$ the locally closed subset of $\wtP \times \wtP$ consisting of pairs $(C,D)$, with $C,D$ distinct smooth plane quartics with $\wtH(C) = \wtH(D)$.  By Theorem~\ref{thm:15}, the monodromy group of the morphism~$\wtH$ is $2$-transitive on fibers and the scheme~$\ppt$ is irreducible.  The pair $(\fer,C_0)$ is in~$\ppt$ and consists of two smooth non-projectively equivalent plane quartics.  By the irreducibility of~$\ppt$, we deduce that the fiber of $\wtH$ over a general point of~$\Pqd$ consists of~$15$ pairwise non-projectively equivalent smooth plane quartics.  In particular, the rational map~$\oh{4}$ is generically finite of the same degree~$15$ as $\wtH$, as stated.
\end{proof}

So far, we proved all the results without using a computer.  The next results, though, involve more lengthier calculations that we find too tedious to check by hand.

\begin{Thm} \label{thm:monodromia}
The monodromy group of~$\wtH$ is the symmetric group~$\mathfrak{S}_{15}$.
\end{Thm}

\begin{proof}
By Theorem~\ref{thm:15}, it suffices to show that the monodromy group of~$\wtH$ contains a transposition.  For this, we exhibit a plane quartic~$D \subset \bpd$ such that $\wtH^{-1}(D)$ is contained in the smooth locus of~$\wtP$ and consists of~$13$ reduced points and a single non-reduced of multiplicity~$2$ (see~\cite{Har}*{Lemma on p.~698}).  Thus, it is sufficient to find a plane quartic~$D$ for which the fiber $\wtH^{-1}(D)$ consists of~$14$ distinct pairs $(C,D)$ with $C \subset \bp$ a smooth plane quartic.  Using the computer algebra program Magma~\cite{magma}, we check that the curve $D$ with equation
\begin{equation} \label{e:kllk}
D \colon \qquad
u^3 (v + w) + v^3 (u + w) + w^3 (u + v) = 0
\end{equation}
has the required properties and we are done.
\end{proof}

\begin{Rem}
In the proof of Theorem~\ref{thm:monodromia}, we saw that above the curve~$D$ of~\eqref{e:kllk} the morphism~$\wtH$ has a unique simple ramification point.  This point corresponds to the smooth plane quartic $Q \subset \bp$ with equation
\[
Q \colon \qquad
\left\{
\begin{array}{rcl}
(x^4 + y^4 + z^4) - 4 (x^3 (y + z) + y^3 (x + z) + z^3 (x + y))
\\[4pt]
+ 6 (x^2 y^2 + x^2 z^2  + y^2 z^2) - 12 xyz (x + y + z)
& = &
0.
\end{array}
\right.
\]
\end{Rem}

\section{A geometric invariant for plane curves}
\label{sec:inva}

Let~$n$ be an even, positive integer; recall that we set~$N+1=\binom{n+2}{2}$.  We define an invariant~$\ro{n}$ of degree~$N+1$ associated to plane curves of degree~$n$.  In the case of plane quartics, we obtain an expression for~$\ro{4}$ in terms of the Dixmier-Ohno invariants.  For background on invariants of plane quartics, we refer to~\cites{dix,ohn}.  We used the package developed in~\cite{lrs} for computations with Dixmier-Ohno invariants.

Let~$R = \Z[a_{ijk}]$ denote the polynomial ring over the integers with~$N+1$ indeterminates, corresponding to the coefficients of the monomials of degree~$n$ in $x,y,z$.  Let $q \in R [x,y,z]_n$ be the universal ternary form $q = \sum a_{ijk} x^i y^j z^k$ of degree~$n$.  We define an $(N+1) \times (N+1)$ symmetric matrix~$\Ro{n}$ with rows and columns indexed by the~$N+1$ monomials of degree~$n$ in $x,y,z$.  The entry of~$\Ro{n}$ corresponding to the pair of monomials $(m_1,m_2)$ is
\[
(\Ro{n})_{m_1,m_2} = \tn ( q , m_1 , m_2 ).
\]

We give two different interpretations for the matrix~$\Ro{n}$.  First, the matrix~$\Ro{n}$ determines a $\Z$-module homomorphism $\Z[x,y,z]_n \to \Z[x,y,z]_n^\vee$ given by $q_1 \mapsto \tn (q,q_1,-)$.  Alternatively, the differential of the map~$\hh{n}$ at the form~$q$ is the linear transformation $\Z[x,y,z]_n \to \Z[u,v,w]_n$ given by $q_1 \mapsto \Jn{n}(q,q_1)$.  Identifying $\Z[u,v,w]_n$ with $\Z[x,y,z]_n^\vee$ via the polar pairing, we obtain that the linear transformation~$\Ro{n}$ is the differential of the map~$\hh{n}$ at~$q$:
\begin{equation} \label{e:dih}
{\textrm{d}}_q \hh{n} = \Ro{n} .
\end{equation}

The determinant of the matrix~$\Ro{n}$ is a polynomial of degree~$N+1$ in the~$N+1$ variables of~$R$.  From either of the two descriptions above, it is clear that~$\det \Ro{n}$ is an invariant for the action of~$SL_{3}$.  We let
\[
\kappa_n = \prod_{\substack{i,j,k \geq 0 \\i+j+k=n}} i!j!k!
\]
be the product of the factorials of all the exponents of all the monomials of degree~$n$ in $x,y,z$; the first few values of~$\kappa_n$ for even~$n$ are
\[
\kappa_2
=
2^3,
\qquad
\kappa_4
=
2^{24} \cdot 3^9,
\qquad
\kappa_6
=
2^{84} \cdot 3^{33} \cdot 5^9,
\qquad
\kappa_8
=
2^{201} \cdot 3^{81} \cdot 5^{30} \cdot 7^9.
\]
We define
\begin{equation} \label{e:rain}
\ro{n} = \frac{1}{\kappa_n}
\det \Ro{n} .
\end{equation}
We deduce that the differential of map~$\hh{n}$ is not an isomorphism at the vanishing set of~$\ro{n}$ and therefore~$\ro{n}$ is an~$SL_3$-invariant.

In the case $n=4$ of ternary quartic forms, the ring of invariants under $SL_3$ is completely explicit.  It is generated by~$13$ invariants, called the~{\emph{Dixmier-Ohno invariants}}:
\begin{itemize}
\item
(Dixmier)
$I_{3d}$, for $d \in \{1,\ldots,7\}$,
\item
(Ohno)
$J_{3d}$, for $d \in \{3,\ldots,7\}$, and
\item
the discriminant $I_{27}$.
\end{itemize}
The indices represent the degree of each invariant as a polynomial in the coefficients of the quartic form.  We follow the notation of~\cite{lrs}.

\begin{Thm} \label{thm:lrs}
The invariant $\ro{4}$ of degree~$15$ satisfies the identity
\[
\frac{2 \cdot 5^4 \cdot 7}{{24}^{15}} \;
\ro{4} =
\left\{
\begin{array}{rrcrr}
2 \cdot   3^3 \cdot 5 \cdot 7^2 &
J_{15}
& - &
2 \cdot   3^3 \cdot 5 \cdot 7&
I_{15} \\
- 3^2 \cdot 5 \cdot 109  &
I_3 J_{12}
& + &
2^3 \cdot 3^5 \cdot 5 &
I_3 I_{12} \\
+ 2   \cdot 3^2 \cdot 137  &
I_3^2 J_9
& + &
3 \cdot   271 &
I_3^2 I_9 \\
+ 2^3 \cdot 3^3 \cdot 5 \cdot 7^2 &
I_6 J_9
& - &
2^4 \cdot 3^3 \cdot 5 \cdot 7^2 &
I_6 I_9 \\
- 2^3 \cdot   5 \cdot 7 \cdot 149 &
I_3^3 I_6
& + &
2^7 \cdot 3^3 \cdot 5 \cdot 7 \cdot 13 &
I_3 I_6^2.
\end{array}
\right.
\]
\end{Thm}

\begin{proof}
The argument is a direct computer calculation.  There are~$11$ monomials of degree~$15$ in the Dixmier-Ohno invariants and the invariant~$\ro{4}$ is a linear combination of these~$11$ monomials.  By choosing~$11$ sufficiently general ternary quartic forms, we check that the identity in the statement of the theorem is the unique solution.  Note that the monomial~$I_3^{5}$ is the unique monomial of degree~$15$ in the Dixmier-Ohno invariants not appearing the expression of~$\ro{4}$.
\end{proof}

\begin{Rem}\label{rem:chp0}
So far, the characteristic of the ground field~$k$ was zero.  Nevertheless, the map~$\hh{4}$ is defined over $\textrm{Spec}\, \Z$ and hence over a field of arbitrary characteristic.  We now assume that the characteristic of the ground field is coprime with~$6$ and we check that the map~$\hh{4}$ is generically finite of degree~$15$.  First, we evaluate the invariant $\ro{4}$ on the quartic form $\kl = x^3 y + y^3 z + z^3 x$, vanishing on the Klein plane quartic.  We obtain $\ro{4} (\kl) = 2^{34} \cdot 3^{24}$, which does not vanish in~$k$.  Therefore, the map~$\hh{4}$ is generically smooth over~$k$ and hence generically finite.  We conclude, by generic flatness, that the degree of~$\hh{4}$ is also~$15$.
\end{Rem}

\appendix

\section{The scheme~$F$ and its degree}
\label{ap:my}

The proof of Lemma~\ref{lem:groeb} appearing in this Appendix is entirely independent of the results of the rest of the paper.  We compute without using the computer the degree of a zero-dimensional scheme~$F$, isomorphic to a scheme that appears in the proof of Theorem~\ref{thm:fibfer}.  The proof could just as well be carried out over the field of rational numbers by a computer algebra system, such as Magma.

Let $\mathbb{P}^6_{k}$ be the projective space over the field~$k$ with homogeneous coordinates $\rho, \sigma_1, \sigma_2, \sigma_3, \tau_1, \tau_2, \tau_3$.  We introduce the subscheme $F$ of $\mathbb{P}^6_{k}$.  Let $\G_0$ be the set of~$14$ forms
\begin{equation} \label{eq:fibfer}
\G_0 =
\left\{
\begin{array}{c}
A_1 = \sigma_{3}^2 - \sigma_{2}^2
\qquad
\qquad
A_2 = \sigma_{3}^2 - \sigma_{1}^2
\\[5pt]
\sigma_{1} \tau_{2}
\qquad
\quad
\sigma_{2} \tau_{1}
\qquad
\quad
\sigma_{3} \tau_{1}
\\[4pt]
\sigma_{1} \tau_{3}
\qquad
\quad
\sigma_{2} \tau_{3}
\qquad
\quad
\sigma_{3} \tau_{2}
\\[5pt]
S_1 = 12 \rho \sigma_{1} + 2 \sigma_{2} \sigma_{3} + \tau_{1}^2
\qquad
\quad
T_1 = 4 \sigma_{1} \tau_{1} - \tau_{2} \tau_{3}
\\[4pt]
S_2 = 12 \rho \sigma_{2} + 2 \sigma_{1} \sigma_{3} + \tau_{2}^2
\qquad
\quad
T_2 = 4 \sigma_{2} \tau_{2} - \tau_{1} \tau_{3}
\\[4pt]
S_3 = 12 \rho \sigma_{3} + 2 \sigma_{1} \sigma_{2} + \tau_{3}^2
\qquad
\quad
T_3 = 4 \sigma_{3} \tau_{3} - \tau_{1} \tau_{2}
\end{array}
\right\} .
\end{equation}
Denote by $F$ the scheme
\begin{equation} \label{eq:fibf}
F \colon V(\G_0) \subset \mathbb{P}^6_{k}
\end{equation}
defined by the vanishing set of~$\G_0$ in~$\mathbb{P}^6_{k}$.

\begin{Lem} \label{lem:groeb}
The scheme $F$ has dimension~$0$ and degree~$15$.  The support of $F$ consists of~$5$ points: the point $[1,0,0,0,0,0,0]$ of multiplicity~$11$ and the~$4$ points $[1,-6,-6,-6,0,0,0]$, $[1,-6,6,6,0,0,0]$, $[1,6,-6,6,0,0,0]$, $[1,6,6,-6,0,0,0]$ of multiplicity~$1$.
\end{Lem}

\begin{proof}
Let $I = \langle \G_0 \rangle$ be the ideal generated by~$\G_0$.  As a first step, we determine a Gr\"obner basis for the ideal~$I$.  Let $\G_1$ be the set of~$4$ forms
\begin{equation} \label{eq:bag}
\G_1 =
\left\{
\begin{array}{c}
\frac{1}{2}
(\sigma_{2} S_1 + 2 \sigma_{3} A_2 - \tau_{1} \cdot \sigma_{2} \tau_{1})
\\[4pt]
\frac{1}{2}
(\sigma_{3} S_1 - \tau_{1} \cdot \sigma_{3} \tau_{1})
\\[4pt]
\frac{1}{2}
(\sigma_{3} S_2 - \tau_{2} \cdot \sigma_{3} \tau_{2})
\\[4pt]
\frac{1}{4}
(\sigma_{3} T_3 + \tau_{2} \cdot \sigma_{3} \tau_{1})
\end{array}
\right\}
=
\left\{
\begin{array}{c}
6 \rho \sigma_{2} \sigma_{1} + \sigma_{3}^3 \\[4pt]
6 \rho \sigma_{3} \sigma_{1} + \sigma_{3}^2 \sigma_{2} \\[4pt]
6 \rho \sigma_{3} \sigma_{2} + \sigma_{3}^2 \sigma_{1} \\[4pt]
\sigma_{3}^2 \tau_{3}
\end{array}
\right\}.
\end{equation}
By construction, the forms in $\G_1$ are contained in the ideal $I$.  Let $\G$ be the set of~$18$ forms $\G = \G_0 \cup \G_1 \subset I$.

Assign the following weights to the variables:
\[
\begin{array}{|c|c|c|c|c|c|c|c|}
\hline
\textrm{Variable:} &
\rho &
\sigma_{3} &
\sigma_{2} &
\sigma_{1} &
\tau_{1} &
\tau_{2} &
\tau_{3} \\[3pt]
\hline
\textrm{Weight:} &
1 &
3 &
4 &
4 &
5 &
5 &
5 \\
\hline
\end{array}
\]
and resolve ties among monomials using the lexicographic ordering with
\[
\rho < \sigma_{3} < \sigma_{2} < \sigma_{1} < \tau_{1} < \tau_{2} < \tau_{3}.
\]
Using Buchberger's Criterion, it is straightforward to check that $\G$ is a Gr\"obner basis of $I$ with respect to the monomial order just defined.  We omit this routine computation.

Let~$I_0$ be the initial ideal of $I$.  Since $\G$ is a Gr\"obner basis of $I$, the monomial ideal $I_0$ is the ideal generated by the~$18$ initial monomials of the elements of $\G$:
\[
I_0 =
\left\langle
\begin{array}{cccccc}
\sigma_{3}^3
&
\sigma_{3}^2 \tau_{2}
&
\sigma_{3}^2 \tau_{1}
&
\sigma_{3}^2 \tau_{3}
&
\sigma_{2}^2
&
\sigma_{1}^2
\\[6pt]
\sigma_{3} \tau_{1}
&
\sigma_{2} \tau_{3}
&
\sigma_{1} \tau_{2}
&
\sigma_{3} \tau_{2}
&
\sigma_{2} \tau_{1}
&
\sigma_{1} \tau_{3}
\\[6pt]
\tau_{1}^2
&
\tau_{2}^2
&
\tau_{3}^2
&
\tau_{1} \tau_{2}
&
\tau_{1} \tau_{3}
&
\tau_{2} \tau_{3}
\end{array}
\right\rangle
.
\]
The~$15$ monomials
\[
\begin{array}{cccccccc}
1
&
\sigma_{1}
&
\sigma_{2}
&
\sigma_{3}
&
\tau_{1}
&
\tau_{2}
&
\tau_{3}
\\[4pt]
\sigma_{3}^2
&
\sigma_{2} \sigma_{3}
&
\sigma_{1} \sigma_{3}
&
\sigma_{1} \sigma_{2}
&
\sigma_{3} \tau_{3}
&
\sigma_{2} \tau_{2}
&
\sigma_{1} \tau_{1}
&
\sigma_{1} \sigma_{2} \sigma_{3}
\end{array}
\]
are all the monomials not divisible by~$\rho$ and not contained in the ideal $I_0$.  Thus, the Hilbert polynomial of the ideal $I_0$ is the constant polynomial $15$, and hence the same is true for the ideal $I$.  We conclude that the scheme $F$ has dimension~$0$ and degree~$15$, as stated.

A direct calculation of the Jacobian of the given equations shows that the points satisfying the inequality $\sigma_{1} \sigma_{2} \sigma_{3} \neq 0$ are reduced points of the scheme $F$.

Denote by $F_{\rm red}$ the reduced subscheme associated to $F$.  We observe that $\tau_1,\tau_2,\tau_3$ vanish on $F_{\rm red}$.  We easily obtain that $F_{\rm red}$ consists of the~$5$ points $[1,-6,-6,-6,0,0,0]$, $[1,-6,6,6,0,0,0]$, $[1,6,-6,6,0,0,0]$, $[1,6,6,-6,0,0,0]$ and $[1,0,0,0,0,0,0]$.  Since the points different from $[1,0,0,0,0,0,0]$ are reduced and the total degree is~$15$, the result follows.
\end{proof}

\bigskip

\small{
\noindent
Marco Pacini, Instituto de Matem\'atica, Universidade Federal Fluminense, \\
Rio de Janeiro, Brazil\\
E-mail: \tt{pacini.uff@gmail.com}, \tt{pacini@impa.br}}

\bigskip

\small{
\noindent
Damiano Testa, Mathematics Institute, University of Warwick,\\
Coventry, CV4 7AL,\\
United Kingdom\\
E-mail: \tt{adomani@gmail.com}}

\end{document}